\documentclass{amsart}
\usepackage{MnSymbol}
\usepackage[margin=2cm]{geometry}
\usepackage{color,xcolor}
\usepackage{amsfonts,amsmath,amsthm}

\usepackage[colorlinks=false,pagebackref,hyperindex]{hyperref}

\hypersetup{
	colorlinks,
	citecolor=violet,
	linkcolor=violet,
	urlcolor=blue}

\newcommand{\FF}{\mathbf F}
\newcommand{\ZZ}{\mathbf Z}
\newcommand{\RR}{\mathbf R}

\newcommand{\NN}{\mathbf N}

\theoremstyle{Theorem}
\newtheorem{thm}{Theorem}[section]
\newtheorem*{mainthm}{Theorem}
\newtheorem{cor}[thm]{Corollary}
\newtheorem{lem}[thm]{Lemma}

\theoremstyle{definition}

\newtheorem{dff}[thm]{Definition}
\newtheorem{xmp}[thm]{Example}
\newtheorem{rmk}[thm]{Remark}

\def\HK{\textnormal{HK}}

\def\fm{\mathfrak{m}}

\def\cH{\mathcal{H}}

\newcommand\sIJe[1]{I^{\lceil sp^{#1} \rceil} + J^{[p^{#1}]}}

\def\dotdiv{\ \hbox{\.{-\!-\!-\ }}}
\def\dotdiv{\, {\begin{picture}(1,1)(-1,-2)\put(2.5,3){\circle*{1.6}}\linethickness{.3mm}\line(1,0){5}\end{picture}}\ \ }

\colorlet{DG}{green!50!black}
\colorlet{DB}{blue!50!black}

\title{The $s$-multiplicity function of $2 \times 2$-determinantal rings}
\author{Lance Edward Miller and William D.\ Taylor}
\begin{document}

\maketitle 

\begin{abstract}
This article generalizes joint work of the first author and I. Swanson to the $s$-multiplicity recently introduced by the second author. For $k$ a field and $X = [ x_{i,j}]$ a $m \times n$-matrix of variables, we utilize Gr\"obner bases to give a closed form the length $\lambda( k[X] / (I_2(X) + \fm^{ \lceil sq \rceil} + \fm^{[q]} ))$ where $s \in \ZZ[p^{-1}]$, $q$ is a sufficiently large power of $p$, and $\fm$ is the homogeneous maximal ideal of $k[X]$. This shows this length is always eventually a {\it polynomial} function of $q$ for all $s$. 
\end{abstract}

\section{Introduction}

\

One of the most well studied and intriguing invariants for positive characteristic commutative algebra is the Hilbert-Kunz multiplicity. Specifically in a local ring $(R,\fm,k)$, where $k$ has positive characteristic, the length $\lambda(R/\fm^{[q]}) = e_{\HK}(R) q^d + O(q^{d-1})$ as was first shown by P. Monsky \cite{Mon83} building on work of E. Kunz. Much subtly lies in the lower order terms. When $R$ is excellent, normal, and with perfect residue field, there is a sharper form $\lambda(R/\fm^{[q]}) = e_{\HK}(R) q^d + \beta q^{d-1} + O(q^{d-2})$ \cite{HMM04}. However, in contrast to the Hilbert-Samuel function, one cannot expect this length to be polynomial in $q$, even for nice rings in small dimensions. Despite its complication, the first author and I. Swanson showed that this length function is a polynomial in $q$ for $R$ the determinantal ring defined by $2$-minors \cite{MS13}. The techniques there in are combinatorial in nature, building on work of K. Eto  and K.-i. Yoshida \cite{EY03}, and were pushed later on by I. Swanson and M. Robinson to give a closed form as a sum of products of binomial coefficients, yielding a complete understanding of the Hilbert-Kunz function of such rings. 

Recently, the second author introduced a type of interpolation between Hilbert-Samuel and Hilbert-Kunz multiplicities. Specifically for $s$ a positive real number, the $s$-multiplicities $e_s(R)$ form a continuous family of real numbers agreeing with the Hilbert-Samuel multiplicity $e(R)$ for small values of $s$ and agreeing with the Hilbert-Kunz multiplicity $e_{\HK}(R)$ for large values of $s$. These arise as suitable normalizations of the limit  $\lim_{q \to \infty} q^{-d} \lambda(R/ (\fm^{\lceil sq \rceil} + \fm^{[q]}))$. which is known to exist \cite{Tay}. This family offers an important hope to deform results from one multiplicity to another. Standing in the way are the multiplicities $e_s(R)$ which agree with neither the Hilbert-Samuel nor Hilbert-Kunz multiplicities, and so far these intermediate values are not well understood.

Fixing $s \in \ZZ[p^{-1}]$, for small values of $s$, the length $\lambda(R/ (\fm^{\lceil sq \rceil} + \fm^{[q]}))$ is eventually a polynomial in $q$ and for large values of $s$ the length can be significantly more complicated. However, when $R$ is the determinantal ring defined by $2$-minors, this length function is eventually a polynomial in $q$ for large values of $s$ too. The purpose of this short article is show in such case, this length function is eventually a polynomial in $q$ for all such $s$ and to give a closed form for it similar to \cite{RS15}. The final form of this is our main theorem, which is a sum of products of binomials and involves the monus operator, denoted $\dotdiv$ and defined by $a\dotdiv b=\max\{a-b,0\}$.  In this theorem and throught the paper, unrestriced sums are interpreted as being over all integers.

\begin{mainthm} [Theorem~\ref{thm:monomialcount}] Fix $k$ a field, $p$ an integer and $q$ a $p$-power. Let $X$ be an $m \times n$-matrix of variables, $\fm$ the homogeneous maximal ideal of $k[X]$ and $I_2(X)$ the ideal of $2 \times 2$-minors. Let $s\in \RR_{>0}$ such that $sq\in \ZZ$, and set
$$R(m,n,s,q) := \sum_{a}\sum_{b}\sum_{\ell}\binom{m-1}{a}\binom{n-1}{b}\binom{sq+\ell}{a+b+1}\binom{a}{\ell}\binom{b}{\ell}$$
 $$S(m,n,s,q) := \sum_{i>0}\sum_{j>0}\sum_{a}\sum_{b}\sum_{\ell}(-1)^{i+j}\binom{m}{i}\binom{n}{j}\binom{(j\dotdiv i)q+m-1}{m-1-a}\binom{(i \dotdiv j)q+n-1}{n-1-b}\binom{(s-\max\{i,j\})q+\ell}{a+b+1}\binom{a}{\ell}\binom{b}{\ell}.$$ The length $\lambda\left(\frac{k[X]}{\mathfrak{m}^{\lceil sq \rceil}+\mathfrak{m}^{[q]}+I_2(X)}\right) = R(m,n,s,q) - S(m,n,s,q)$. Notably, this is a polynomial in $q$. 
\end{mainthm}


{\bf Acknowledgments:} We thank P.\ Mantero and M.\ Johnson for helpful discussions and D. Juda for suggesting the technique used in one of the proofs of Lemma~\ref{product of binom coeff}. We deeply thank the referee for the careful and detailed review of the manuscript, which allowed us to improve it significantly. 

\section{Preliminaries}

Unless otherwise stated, $p$ always denotes a positive prime integer, $q$ a power of $p$, and $k$ a field of characteristic $p$. Throughout $s$ is a positive real number and $\lambda$ denotes length of a module.  The $s$-multiplicity, introduced in \cite{Tay}, is defined as follows. Fix a local ring $(R,\fm)$ of characteristic $p$ and two $\fm$-primary ideals $I$ and $J$, the following limit \cite[Thm. 2.1]{Tay} exists,  $$h_s(I,J) := \lim_{e \to \infty} \lambda( R / \sIJe{e} )/p^{ed}.$$ For small values of $s$, $h_s(I,J) = \frac{s^d}{d!} e(I)$ whereas for large values of $s$,  $h_s(I,J) = e_{\HK}(J)$. When $R$ is regular of dimension $d$, $\cH_s(d) := h_s(\fm,\fm) = \sum_{i=0}^{ \lfloor s \rfloor} \frac{ (-1)^i}{d!} \binom{d}{i} (s-i)^d$ offers a normalizing factor and one defines the $s$-multiplicity by $e_s(I,J) := h_s(I,J)/\cH_s(d)$.  We follow the usual conventions denoting $e_s(R) := e_s(\fm,\fm)$ and similarly for $h_s$. 

This article concerns the $s$-length functions $h_s(R)$ where $R$ is the quotient of a polynomial ring with defining ideal the $2 \times 2$-minors of a matrix of variables. The techniques follow similarly to \cite{MS13,RS15}. We first recall notation.

\begin{dff}\label{defstaircase}
We call a monomial $\prod_{i,j} x_{i,j}^{p_{i,j}}$
a {\bf staircase monomial}
if whenever $i < i'$ and $j < j'$,
then $p_{i,j} p_{i',j'} = 0$.
A staircase monomial is called a {\bf stair monomial}
if there exist $c \in \{1, \ldots, m\}$ and $d \in \{1, \ldots, n\}$
such that $p_{l,k} = 0$ whenever $(l-c)(k-d) \not = 0$.
Thus the indices $(i,j)$ for which $p_{i,j} \not = 0$
all lie in the union of part of row $c$ with part of column $d$,
either in a $\lefthalfcap$ or a $\righthalfcup$ configuration. A stair monomial
is called a {\bf $\bf q$-stair monomial} if for such $c, d$, $\sum_k p_{c,k} = q = \sum_k p_{k,d}$.
\end{dff}

\begin{rmk} Notice that monomials $p$ in $k[X]$ may be identified with integer valued $m \times n$-matrices by writing $p = \prod x_{i,j}^{p_{i,j}}$ and associating $p$ to $(p_{i,j})$. We call this its {\bf exponent matrix}. Staircase monomials $\prod_{i,j} x_{i,j}^{p_{i,j}}$ are so called as the indices $(i,j)$ for which $p_{i,j} \not = 0$ lie on a southwest-northeast staircase type pattern, i.e., their exponent matrices have support in a pattern like the following 
$$\left[
\begin{matrix}
 & & & & & \bullet & \bullet & \bullet \cr
 & & & & & \bullet & \cr
 & & & \bullet & \bullet & \bullet \cr
 & & & \bullet & \cr
 \bullet & \bullet & \bullet & \bullet \cr
\end{matrix}\right].
$$

Under this identification, the multiplicative semigroup of monomials is identified with the additive semigroup of non-negative integer valued matrices. We tacitly use this identification to keep the notation in the proofs to a minimum.
\end{rmk}

\noindent We start with an elementary lemma about staircase monomials implicit in the work \cite{MS13, RS15}. 

\begin{lem}\label{lem:minordivis} Let $X$ be a generic $m \times n$-matrix. 
\begin{enumerate}
\item Any monomial in $k[X]$ is equivalent to a staircase monomial modulo $I_2(X)$. 
\item If $p$ is a monomial and $q$ is a staircase monomial with $p \equiv q \bmod I_2(X)$, then $p$ has the same degree, row, and column sums as $q$. 
\end{enumerate}
\end{lem}
\begin{proof}

Let $p$ be a monomial in $k[X]$, identified with its exponent matrix $(p_{i,j})$. The key mechanic at work here is that when $a < b$ and $c < d$, modulo the minor $x_{a,c} x_{b,d} - x_{a,d} x_{b,c}$, the monomial $p$ is equivalent to the monomial $p'$ with exponent matrix $(p'_{i,j})$ where $$p'_{i,j} = \begin{cases} p_{i,j} & (i,j) \neq (a,c), (b,d), (b,c), (a,d) \\ 
p_{a,c} + 1 & (i,j) = (a,c)\\
p_{b,d} + 1 & (i,j) = (b,d)\\
p_{a,d} - 1 & (i,j) = (a,d)\\
p_{b,c} - 1 & (i,j) = (b,c)\\
 \end{cases}$$ From this the second claim is immediate as modifying monomials using these determinants clearly preserves all listed characteristics.

To prove the first claim, a simple induction on the number of columns allows us to assume that any monomial in correspondence to the augmented $m \times n$-matrix $[(p_{i,j})_{ 1 \leq j \leq n-1} | 0 ]$ is equivalent modulo $I_2(X)$ to a staircase monomial, that is modulo $I_2(X)$ we may assume $p$ has the staircase shape for the first $n-1$ columns. Set $i$ to be the smallest row index so that $p_{i, n-1} \neq 0$. We now induce on the row indicies $i'$ such that $i \leq i'$ and $p_{i',n} \neq 0$. If there are none or if the only one is $i' = i$, then $p$ is a already staircase monomial. Otherwise, assume by induction that $p$ is equivalent modulo minors to a staircase monomial in the support of the last column of exponent matrix is in rows $1$ through $i' -1$ and that $p_{i', n} \neq 0$

If $p_{i', j} = 0$ for all $j$, then up to a multiple of the minor $x_{i, n-1} x_{i', n} - x_{i', n-1} x_{i,n}$, $p$ is equivalent to a monomial with $p_{i, n-1} p_{i', n} = 0$ and the rest follows by induction. Otherwise, we may set $j$ to be the largest column index for which $p_{i',j} \neq 0$. If the value $p_{i', n}$ is larger than $\sum p_{k,\ell}$ where $i \leq k \leq i', j \leq \ell \leq n$ but $(k,\ell) \neq (i',n)$, then using appropriate minors, one may assume that $p$ is equivalent to a monomial with the same support as $p$ outside of the ranges $i \leq k \leq i'$ and $j \leq \ell \leq n$ but where $p_{k, \ell} = 0$ for $i \leq k \leq i'$ and $j \leq \ell \leq n$ unless either $k = i'$ or $j = n$, i.e., is a staircase monomial. If $\sum p_{k,\ell}$ is smaller than $p_{i',n}$, then again we may use minors to reduce $p$ modulo $I_2(X)$ to a monomial which is staircase up to the last column but for which $p_{i', n} = 0$ and the rest follows by induction. 
\end{proof}

\begin{thm}\label{thm:mainGB}
Fix $X$ a generic $m \times n$-matrix and $\fm$ the homogeneous maximal ideal in $k[X]$. For $I$ an $\fm$-primary monomial ideal in $k[X]$, the set of staircase monomials not in $I$ forms a $k$-basis for $k[X]/(I_2(X) + I)$. 
\end{thm}
\begin{proof}

It suffices to show the claim for $I = \fm^t$ for fixed $t$ as the theorem follows by noting that when $I$ is $\fm$-primary, $\fm^t \subset I$ for some power $t$ and a $k$-basis for $k[X]/(I_2(X) + I)$ is obtained by eliminating from a $k$-basis of $k[X]/(I_2(X) + \fm^t)$ the elements which lie in $I$.  

Set $T \subset k[X]$ the ideal generated by all staircase monomials of degree $t$. Following the proof of \cite[Thm. 2.4]{MS13}, the theorem follows immediately once we've show that $G = I_2(X) + T \subset I_2(X) + \fm^t$ is a Gr\"obner basis. The equality $I_2(X) + T = I_2(X) + \fm^t$ follows by the first claim of Lemma~\ref{lem:minordivis}. To finish the proof, one needs only check, via the Buchberger algorithm, that $S$-polynomials $S(f,g)$ for any generators $f$ and $g$ of $I_2(X) + T$. This is immediately trivial unless $f$ is a determinant, $g$ is a staircase monomial, and their leading terms share a variable in common. The rest of the check is straightforward and follows by repeating the same case by case analysis as in the proof of \cite[Thm. 2.2]{MS13}. 
\end{proof}

\begin{rmk}
In the special case that $I = \fm^{[q]}$ one notes that the basis guaranteed by Theorem~\ref{thm:mainGB} agrees with base guaranteed by \cite[Thm. 2.2]{MS13} as the set of staircase monomials which are not in $I$ are precisely those not divisible by any $q$-stair monomials. 
\end{rmk}

\begin{cor}\label{cor:GBMaxIdeal}
Fix $X$ a generic $m \times n$-matrix and $\fm$ the homogeneous maximal ideal in $k[X]$. For any positive integer $s$, a $k$-basis for $\fm^{ \lceil sq \rceil} + \fm^{[q]} + I_2(X)$ consists of staircase monomials of degree at most $\lceil sq \rceil$ and having either all row sums smaller than $q$ or all column sums column sums smaller than $q$. 
\end{cor}

It now suffices for us to turn our attention at carefully counting the $k$-basis of $\fm^{ \lceil sq \rceil} + \fm^{[q]} + I_2(X)$. The computation of $\lambda\left(\frac{k[X]}{\mathfrak{m}^{\lceil sq \rceil}+\mathfrak{m}^{[q]}+I_2(X)}\right)$ follows an expected combinatorial argument. Following the techniques in \cite{RS15}, we explain how to give a precise enough monomial count. 

In the rest of the paper we will be operating under the assumptions that $s\in \ZZ[p^{-1}]$ and $q$ is large enough that $sq\in \ZZ$.  We do this because we are primarily interested in establishing that the length function is polynomial in $q$. It is unreasonable to expect such behavior for $s\notin \ZZ[p^{-1}]$.  As a simple example, consider the ring $R=\FF_2[x,y]$ and let $\fm=(x,y)$ and $s=\frac{4}{3}$.  For any $e\in \NN$, we have that
$$\lceil sp^e\rceil=\left\lceil \frac{2^{e+2}}{3}\right\rceil =\begin{cases} \frac{2^{e+2}+1}{3} & \text{if }e\text{ is odd}\\ \frac{2^{e+2}+2}{3} & \text{if }e\text{ is even.}\end{cases}$$
From this we can easily compute the length function in question:
$$\lambda(R/\fm^{\lceil sp^e\rceil}+\fm^{[p^e]})=\begin{cases} \frac{7}{9}p^{2e}+\frac{5}{9}p^e-\frac{2}{9}  &\text{if }e>0\text{ is odd}\\\frac{7}{9}p^{2e}+\frac{7}{9}p^e-\frac{5}{9}  &\text{if }e>0\text{ is even.}\end{cases}$$
This example shows that even in the simplest cases, we cannot expect the length function to be equal to a single polynomial when $s\notin\ZZ[p^{-1}]$. 

\section{Combinatorics} We utilize the convention that $\binom{m}{n}=0$ if $n<0$, $m<n$, or $m<0$.  Unspecified summations are over all integers. We are interested in counting staircase monomials with restricted row and column sums. Using \cite[Lem. 2.4]{RS15}, it suffices to count $(m+n)$-tuples $(x_1,\ldots, x_m,y_1,\ldots, y_n)\in \ZZ_{\geq 0}^{m+n}$, where we interpret the $x_i$'s as row sums and $y_j$'s as column sums of the associated exponent matrix to the staircase monomial. This forces the condition $\sum_i x_i = \sum_j y_j$ and the lemma, {\it loc.\ cit.}, gives a bijection between such tuples and staircase monomials. Thus by Corollary~\ref{cor:GBMaxIdeal}, to calculate the length of $k[X]/(I_2(X) + \fm^{ \lceil sq \rceil} + \fm^{[q]})$, it suffices to count all $(m+n)$-tuples $(x_1,\ldots, x_m,y_1,\ldots, y_n)\in \ZZ_{\geq 0}^{m+n}$ such that $\sum_i x_i = \sum_j y_j < \lceil sq \rceil$ and either all $x_i < q$ for all $1 \leq i \leq m$ or all $y_j < q$ for all $1 \leq j \leq q$. There is a natural symmetry to this requirement which we exploit via an inclusion-exclusion type argument. To this end, we introduce two symbols $T$ and $U$ which count monomials meeting relevant conditions.

\begin{dff} Fix $m, n,r$ and $q$ in $\NN$. Let $T(m,n,r,q)$ be the number of $m+n$-tuples, \newline $(x_1,\ldots, x_m,y_1,\ldots, y_n)\in \ZZ_{\geq 0}^{m+n}$ such that 

$$\sum_i x_i=\sum_j y_j<r \text{ and } x_i<q \text{  for all }1 \leq i \leq m.$$

Also let $U(m,n,r,q)$ of all $m+n$-tuples $(x_1,\ldots, x_m,y_1,\ldots, y_n) \in \ZZ_{\geq 0}^{m+n}$ such that 
$$\sum_i x_i=\sum_j y_j<r \text{, for all }i, x_i<q \text{, and for all } j, y_j < q.$$ 
\end{dff}

\begin{rmk} The functions $T(m,n,r,q)$ and $U(m,n,r,q)$ were utilized in Eto and Yoshida's calculation of the Hilbert-Kunz multiplicity of the determinantal ring defined by $2$-minors, viewed as the Segre product of two polynomial rings \cite{Eto02, EY03}. Specifically, for $X$ an $m\times n$-matrix of variables, the ring $k[X]/I_2(X)$ is isomorphic to $k[z_1,\ldots,z_m] \# k[y_1,\ldots,y_n]$. From \cite[Rmk. 2.5]{MS13}, one expresses the length using the following monomial counts. Set $\alpha_{m,d}$ the number of monomials in $k[z_1,\ldots,z_m]$ of total degree $d$ and $\alpha_{m,d,q}$ the number of monomials in $k[z_1,\ldots,z_m]$ of total degree $d$ and $z_i$-degree at most $q$ for all $i$, and similarly for $k[y_1,\ldots,y_m]$. From \cite[Rmk. 2.5]{MS13} one expresses the length via inclusion-exclusion
$$\lambda( k[X] / (I_2(X) + \fm^{[q]}) = \sum_{d=0}^{ (q-1)n} \alpha_{m,d} \alpha_{n,d,q} + \sum_{d=0}^{(q-1)m} \alpha_{n,d}\alpha_{m,d,q} - \sum_{d = 0}^{(q-1)m} \alpha_{n,q,d} \alpha_{m,d,q}.$$
Immediately one has 
$$\lambda( k[X] / (I_2(X) + \fm^{ \lceil sq \rceil } + \fm^{[q]})) =  \sum_{d=0}^{ \lceil sq \rceil } \alpha_{m,d} \alpha_{n,d,q} + \sum_{d=0}^{ \lceil sq \rceil } \alpha_{n,d}\alpha_{m,d,q} - \sum_{d = 0}^{\lceil sq \rceil } \alpha_{n,q,d} \alpha_{m,d,q}.$$ The functions $T(n,m,r,q)$ and $U(n,m,r,q)$ arise from exploiting the correspondence in \cite[Lem. 2.4]{RS15} between monomials and tuples. 
\end{rmk}

Our goal is to give a closed form for $T(m,n,r,q)$ and $U(m,n,r,q)$. We start with a helpful auxiliary combinatorial identity.

\begin{lem}\label{product of binom coeff} For $a,b,c\in\NN$, 
$\displaystyle{\sum_{w=0}^{\infty} \binom{c+w}{a+b}\binom{a}{w}\binom{b}{w}=\binom{c}{a}\binom{c}{b}}$.
\end{lem}

We offer two proofs of this statement. The first is based on the Zeilberger-Wilf algorithm. The second is a more elaborate combinatorial proof which realizes a bijection between two sets each of which obviously having cardinalities both sides of this identity. 

\begin{proof}(of Lemma~\ref{product of binom coeff})
Set $F(w,a) = \binom{c+w}{a+b}\binom{a}{w}\binom{b}{w}$ and $$G(w,a) = \frac{ w^2( -a-b+c+w) }{ (1+a+b) ( -1 -a + w) } F(w,a).$$ One may immediately verify the identity \begin{equation}\label{eqZW}G(w+1,a) - G(w,a) = (a-c) F(w,a) + (1+a) F(w, a+1).\end{equation} Setting $H(a) := \sum_w F(w,a)$ and summing \eqref{eqZW} over $w$ we have $0 = (a-c) H(a) + (1+a) H(a+1),$ and thus $$H(a) = \frac{ (c-a+1) \cdots (c) }{ a! } H(0) = \binom{c}{a} \sum_{w}^0 \binom{c+w}{b} \binom{b}{w} = \binom{c}{a} \binom{c}{b}.$$ 
\end{proof}

Next, we give a stronger combinatorial proof of Lemma~\ref{product of binom coeff}. We introduce some notation only used for this proof. For $n\in\mathbf{N}$, let $[n]$ denote the set $\{1,2,\ldots, n\}$.  By a \emph{colored integer} we mean an element of $[c]\times\{\text{red},\text{blue}\}$.  We call the first component of a colored integer its \emph{value} and we call the second component its \emph{color}. We impose an order on the set of colored integers by declaring that $\text{red} < \text{blue}$ and using the lexicographic order. In particular, $x<x'$ if either the value of $x$ is less than the value of $x'$, or their values are equal, $x$ is red, and $x'$ is blue. For example, $2 \times \text{blue} < 3 \times \text{red}$ and $5 \times \text{red} < 5 \times \text{blue}.$ By a \emph{chain of type $(a,b)$} we mean a chain of colored integers $x_1<\cdots <x_{a+b}$ containing $a$ red integers and $b$ blue integers. Such a chain is determined completely by the values of the red integers and the values of the blue integers, and so the number of chains of type $(a,b)$ is $\binom{c}{a}\binom{c}{b}$.

\begin{proof}(of Lemma~\ref{product of binom coeff}) We show that the number of chains of type $(a,b)$ is equal to the number of $4$-tuples $(w,A,B,C)$, where $w$ is an integer, $A\subseteq [a]$ has size $w$, $B\subseteq [b]$ has size $w$, and $C\subseteq [c+w]$ has size $a+b$. The latter set clearly has size $\sum_{w=0}^{\infty}\binom{c+w}{a+b}\binom{a}{w}\binom{b}{w}$ and so finding such a bijection immediately establishes the desired equality. All sets involved are totally ordered, so we utilize the notation $\{ i_1 < \ldots < i_n \}$ for a set of natural numbers $i_1, \ldots, i_n$ ordered as indicated.

We first define a function $\varphi$ from the set of all chains of type $(a,b)$ to the set of $4$-tuples. To do so, we need one more piece of terminology. We call a consecutive pair of colored integers $(x_i,x_{i+1})$ with $x_i < x_{i+1}$ an \emph{rb-pair} if $x_i$ is red and $x_{i+1}$ is blue. Note by the ordering there are two types of rb-pairs, those with equal value and those with differing values. We call an rb-pair \emph{stable} provided the values in the pair agree. 

Let $X$ be a chain of type $(a,b)$ consisting of colored integers $x_1<\cdots <x_{a+b}$.  Let $\{i_1<\cdots <i_a\}$ be the set of indices of red integers in $X$ and let $\{j_1<\cdots <j_b\}$ be the set of indices of blue integers in $X$.  We first encode the rb-pairs.  Set $A=\{\ell : (x_{i_\ell}, x_{i_\ell + 1}) \text{ is an rb-pair}\}$ and similarly $B=\{\ell : (x_{j_\ell-1},x_{j_\ell}) \text{is an rb-pair}\}$.  Clearly  $A\subseteq [a]$ and $B\subseteq [b]$, and $\# A = \# B$.  

We now produce the tuple $(w,A,B,C)$. The sets $A$ and $B$ have already been defined and both have cardinality $w$. It suffices now to construct $C$. This will encode both the values of the chain and the locations of the stable rb-pairs. Set $C'\subseteq [c]$ be the set of values of the elements of $X$. To capture the location of the stable rb-pairs in a recoverable manner, we write $A = \{ s_1 < \ldots < s_w \}$ and set 
 $$C''=\{\ell : \text{the rb-pair starting with } x_{s_\ell} \text{ is stable}\}.$$  One may check that $\# C'+ \# C''=a+b$. Setting $C=C'\cup\{c+\ell : \ell\in C''\}\subseteq [c+w]$, we have produced from the given chain $X$ of type $(a,b)$ a tuple $\varphi(X) := (w,A,B,C)$. Its clear that $\varphi$ is injective. In particular, two chains $X$ and $X'$ have $\varphi(X) = \varphi(X')$, then they must have the same value set as both are recovered by $C \cap [c]$, and the same locations of rb-pairs as both are determined by $A$ and $B$, as well as the same locations of stable rb-pairs as these locations are determined by $C \setminus (C \cap [c])$. All this data completely determines the red and blue colored integers in the chains $X$ and $X'$, hence they are the same chain.

We now check that $\varphi$ is surjective. Fix $(w,A,B,C)$ a $4$-tuple of the desired form. Write $A=\{i_1<\cdots <i_w\}\subseteq [a]$ and $B=\{j_1<\cdots <j_w\}\subseteq [b]$. Also decompose $C'=C\cap [c]=\{c_1<\cdots <c_{a+b-f}\}$ and let $C''=\{s_1<\cdots <s_f\}$ where $C\cap [c+1,c+j]=\{c+s_1<\cdots < c+s_f\}$. 
We now build a chain $X=(x_1<\cdots <x_{a+b})$ of type $(a,b)$ with $\varphi(X) = (w,A,B,C)$. To determine the chain we first construct the coloring, that is we describe a sequence of $a+b$ colored buckets into which we will place values. This is determined by the sets $A$ and $B$. Color the first $j_1-1$ buckets blue, then the next $i_1$ buckets red, the next $j_2-j_1$ buckets blue, the next $i_2-i_1$ buckets red, and so on. Finishing this, the sequence may be too short, however we know that there will be no more rb-pairs, so we fill in with the remaining number of blue buckets, then the remaining number of red buckets. 

Now it suffices to fill in the values. The coloring has been chosen so that the $i_\ell$th red bucket is part of an rb-pair for  $1\leq \ell \leq w$, and similarly the $j_\ell$th blue bucket is part of an rb-pair. Use the set $C''$ to mark the red component of the stable rb-pairs. Now start placing values in buckets in order, and repeat values on the stable rb-pairs so marked. This produces the chain $X$. To see that $\varphi(X) = (w,A,B,C)$ note that $A$ and $B$ characterize the rb-pairs of $X$ and $C$ consists precisely of the values and the encoded locations of the stable rb-pairs by construction.

\end{proof}

\begin{xmp} Fix $a  = 7, b = 8$, and $c = 15$. As is typical with combinatorial proofs, it is instructive to see the functions in action in an example. Consider the $(7,8)$ chain $$1r < 2r < 3r < 4b < 5r < 5b < 6b < 7b < 8b < 9r < 10r < 10b < 11b < 12r < 13b$$ where we denote a red number $nr$ with value $n$ and a blue number $nb$ with value $n$. Realize this chain as $x_1 < \ldots < x_{15}$ colored integers. 

The red indicies are $\{ 1 < 2 < 3 < 5 < 10 < 11 < 14 \}$ and the blue indicies are $\{ 4 < 6 < 7 < 8 < 9 < 12 < 13 < 15 \}$. Writing the former set as $\{ i_1 < \ldots < i_7 \}$ and the latter as $\{ j_1 < \ldots < j_8 \}$, we have the set $A$ of red subindicies of rb-pairs is $\{ 3 < 4 < 6 < 7 \}$ and $B$ the set of blue subindicies of rb-pairs is $\{ 1 < 2 < 6 < 8 \}$. So $w = 4$. The set $C'$ is the set of values $\{ 1, \ldots, 13\}$. The set $C''$ is the set of those indices in $A$ which arise for stable rb-pairs, in this case $C'' = \{ 2, 3\}$. Shifting these by $c = 15$ we have $C = \{ 1, \ldots, 13, 17, 18 \}$ and we have $\varphi(X) = (4, A, B, C)$. 

Continuing with $a = 7, b= 8$, and $c = 15$, we compute $\psi( 2, A, B, C)$ where $A = \{ 3 < 5 \}$, $B = \{ 1 < 2 \}$, and $C = \{ 1, \ldots, 14,17 \} \subset [17]$.  To calculate $\psi( 2, A, B, C)$ we split $C$ into the honest values $\{ 1, \ldots, 14 \}$ and the index $17 - 15 = 2$ which corresponds to a unique stable rb-pair. 
We first determine the pattern of colors.  Since $B = \{ 1 < 2 \}$, the first blue number is part of an rb-pair, which means the sequence starts with red numbers.  Since $A = \{ 3 < 5 \}$, the third red number is the earliest one that is part of an rb-pair, so our sequence starts $r < r < r < b$.  The second blue number is also part of an rb-pair, and so we must switch back to red numbers until we reach the 5th red number, so our sequence looks like $r < r < r < b < r < r < b$.  There are no more rb-pairs, and so we must finish writing blue numbers and then end with the remaining red numbers: $r < r <r < b <r <r <b <b< b< b< b< b< b < r < r$.  We have that the values of the elements in our chain are nonincreasing, and the 2nd rb-pair is the only stable rb-pair, and the values of the 15 colored integers include all values in $[14]$, hence our chain of type $(7,8)$ is 
\[1r<2r<3r<4b<5r<6r<6b<7b<8b<9b<10b<11b<12b<13r<14r.\] 

\end{xmp}

With Lemma~\ref{product of binom coeff} in hand, we draw out a few immediate consequences, which will be applied in the main counting result, Theorem~\ref{thm:TCount}. 

\begin{cor}\label{sum of products of binom coeff}  For $a,b,c\in\NN$,  $\displaystyle{\sum_{i=0}^c \binom{i}{a}\binom{i}{b}=\sum_{j}\binom{c+j+1}{a+b+1}\binom{a}{j}\binom{b}{j}}$.
\end{cor}

\begin{proof} By Lemma~\ref{product of binom coeff} and the Hockeystick Lemma,
\[\sum_{i=0}^c \binom{i}{a}\binom{i}{b}=\sum_{i=0}^c \sum_{j}\binom{i+j}{a+b}\binom{a}{j}\binom{b}{j}
=\sum_{j}\left(\sum_{i=0}^c\binom{i+j}{a+b}\right)\binom{a}{j}\binom{b}{j}
=\sum_{j}\binom{c+j+1}{a+b+1}\binom{a}{j}\binom{b}{j}.\qedhere\]
\end{proof} 

\begin{cor} \label{general sum of products of binom coeff} For $c,u,u',v,v'\in \NN$,  $\displaystyle{\sum_{i=0}^c \binom{t+i}{u}\binom{v+i}{w}=\sum_{a}\sum_{b}\sum_{j}\binom{t}{u-a}\binom{v}{w-b}\binom{c+j+1}{a+b+1}\binom{a}{j}\binom{b}{j}}$.
\end{cor}

\begin{proof} The sum $\sum_{i=0}^c \binom{t+i}{u}\binom{v+i}{w}$ is the coefficient of $x^uy^w$ in the polynomial 
\[\sum_{i=0}^c(x+1)^{i+t}(y+1)^{i+v}=(x+1)^{t}(y+1)^{v}\sum_{i=0}^c(x+1)^{i}(y+1)^i.\]
The coefficient of $x^uy^w$ in the right hand side is, by Corollary \ref{sum of products of binom coeff},
\[\sum_{a}\sum_{b}\binom{t}{u-a}\binom{v}{w-b}\sum_{i=0}^c\binom{i}{a}\binom{i}{b}=\sum_{a}\sum_{b}\sum_{j}\binom{t}{u-a}\binom{v}{w-b}\binom{c+j+1}{a+b+1}\binom{a}{j}\binom{b}{j}.\qedhere\]
\end{proof}

The final ingredient is following lemma, which offers a direct count of the type of tuples we are interested in. Its proof is a direct application of \cite[Lem. 2.5]{RS15}, inclusion-exclusion, and Pascal's identity and left to the reader. 

\begin{lem}\label{lem:UCount} Fix natural numbers $v, d$, and $q$. The number of tuples $(z_1,\ldots, z_v) \in \ZZ_{\geq 0}^v$ with $\sum_i z_i = d$ and $z_i < q$ is $$\sum_i (-1)^i\binom{v}{i}\binom{d-iq+v-1}{v-1}.$$
\end{lem}

Armed with this, we obtain a closed form for $T(m,n,r,q)$ and $U(m,n,r,q)$. These closed forms involve the monus operation $a \dotdiv b = \max\{ a -b, 0\}$. 

\begin{thm}\label{thm:TCount} For fixed $d\in\NN$,
\[T(m,n,r,q)=\sum_{i}\sum_{a}\sum_{b}\sum_{j}(-1)^i \binom{m}{i}\binom{m-1}{m-1-a}\binom{iq+n-1}{n-1-b}\binom{r-iq+j}{a+b+1}\binom{a}{j}\binom{b}{j}\] and 
\[U(m,n,r,q)=\sum_i\sum_{j}\sum_{a}\sum_{b}\sum_{\ell}(-1)^{i+j}\binom{m}{i}\binom{n}{j}\binom{(j\dotdiv i)q+m-1}{m-1-a}\binom{(i \dotdiv j)q+n-1}{n-1-b}\binom{r-\max\{i,j\}q+\ell}{a+b+1}\binom{a}{\ell}\binom{b}{\ell}\]
\end{thm}
\begin{proof}

Both claims are proved using similar techniques. 
By Lemma~\ref{lem:UCount}, the number of $m$-tuples $(x_1,\ldots,x_m)$ such that $\sum_i x_i = d$ and $x_i<q$ for all $i$ is $$\displaystyle{\sum_i(-1)^i \binom{m}{i}\binom{d-iq+m-1}{m-1}}$$ and the number of $n$-tuples $(y_1,\ldots,y_n)$ with $\sum_j y_j=d$ is $\binom{d+n-1}{n-1}$.  Therefore, 
\[T(m,n,r,q)=\sum_{d=0}^{r-1}\sum_{i}(-1)^i \binom{m}{i}\binom{d-iq+m-1}{m-1}\binom{d+n-1}{n-1}
=\sum_{i}(-1)^i \binom{m}{i}\sum_{d=0}^{r-1} \binom{d-iq+m-1}{m-1}\binom{d+n-1}{n-1}. \]
Applying Corollary~\ref{general sum of products of binom coeff} with $c=r-iq-1$, $t=u=m-1$,  $v=iq+n-1$, and $w=n-1$, we obtain that

\[\sum_{d=0}^{r-1} \binom{d-iq+m-1}{m-1}\binom{d+n-1}{n-1}
=\sum_{d'=0}^{r-iq-1}\binom{d'+m-1}{m-1}\binom{d'+iq+n-1}{n-1}
=\sum_{a}\sum_{b}\sum_{j}\binom{m-1}{m-1-a}\binom{iq+n-1}{n-1-b}\binom{r-iq+j}{a+b+1}\binom{a}{j}\binom{b}{j}.
\]
Thus,
\[T(m,n,r,q)=\sum_{i}\sum_{a}\sum_{b}\sum_{j}(-1)^i \binom{m}{i}\binom{m-1}{m-1-a}\binom{iq+n-1}{n-1-b}\binom{r-iq+j}{a+b+1}\binom{a}{j}\binom{b}{j}.\]
Similarly, we find an equivalent expression for $U$.  We have that
\[U(m,n,r,q)=\sum_{d=0}^{r-1}\sum_i\sum_{j}(-1)^{i+j}\binom{m}{i}\binom{n}{j}\binom{d-iq+m-1}{m-1}\binom{d-jq+n-1}{n-1}.\]
if $i\geq j$, then letting $c=r-iq-1$, $t=u=m-1$, $v=iq-jq+n-1$, and $w=n-1$, we have that
\begin{align*}
\sum_{d=0}^{r-1}\binom{d-iq+m-1}{m-1}\binom{d-jq+n-1}{n-1}
&=\sum_{d'=0}^{r-iq-1}\binom{d'+m-1}{m-1}\binom{d'+iq-jq+n-1}{n-1}\\
&=\sum_{a}\sum_{b}\sum_{\ell}\binom{m-1}{m-1-a}\binom{iq-jq+n-1}{n-1-b}\binom{r-iq+\ell}{a+b+1}\binom{a}{\ell}\binom{b}{\ell}.\end{align*}
By a symmetric argument, if $i<j$, then
\[\sum_{d=0}^{r-1}\binom{d-iq+m-1}{m-1}\binom{d-jq+n-1}{n-1}
=\sum_{a}\sum_{b}\sum_{\ell}\binom{jq-iq+m-1}{m-1-a}\binom{n-1}{n-1-b}\binom{r-jq+\ell}{a+b+1}\binom{a}{\ell}\binom{b}{\ell}.\]
Therefore,
\[U(m,n,r,q)=\sum_i\sum_{j}\sum_{a}\sum_{b}\sum_{\ell}(-1)^{i+j}\binom{m}{i}\binom{n}{j}\binom{(j\dotdiv i)q+m-1}{m-1-a}\binom{(i \dotdiv j)q+n-1}{n-1-b}\binom{r-\max\{i,j\}q+\ell}{a+b+1}\binom{a}{\ell}\binom{b}{\ell}.\]
\end{proof}

We are now set to put this all together to give a closed form for the desired length function. 

\begin{thm}\label{thm:monomialcount} Fix $k$ a field, $p$ an integer and $q$ a $p$-power. Let $X$ be an $m \times n$-matrix of variables, $\fm$ the homogeneous maximal ideal of $k[X]$ and $I_2(X)$ the ideal of $2 \times 2$-minors. Let $s\in \RR_{>0}$ such that $sq\in \ZZ$, and set
$$R(m,n,s,q) := \sum_{a}\sum_{b}\sum_{\ell}\binom{m-1}{a}\binom{n-1}{b}\binom{sq+\ell}{a+b+1}\binom{a}{\ell}\binom{b}{\ell}$$
 $$S(m,n,s,q) := \sum_{i>0}\sum_{j>0}\sum_{a}\sum_{b}\sum_{\ell}(-1)^{i+j}\binom{m}{i}\binom{n}{j}\binom{(j\dotdiv i)q+m-1}{m-1-a}\binom{(i \dotdiv j)q+n-1}{n-1-b}\binom{(s-\max\{i,j\})q+\ell}{a+b+1}\binom{a}{\ell}\binom{b}{\ell}.$$ The length $\lambda\left(\frac{k[X]}{\mathfrak{m}^{\lceil sq \rceil}+\mathfrak{m}^{[q]}+I_2(X)}\right) = R(m,n,s,q) - S(m,n,s,q)$ is eventually a polynomial in $q$ for all $s$.
\end{thm}
\begin{proof}
First use Corollary~\ref{cor:GBMaxIdeal} to give a $k$-basis for the vector space $k[X] / (\mathfrak{m}^{\lceil sq \rceil }+\mathfrak{m}^{[q]}+I_2(X))$ consisting of monomials of bounded degree and with restricted row and column sums. This reduces the calculation to the functions $T$ and $U$. In particular, by applying Lemma~\ref{lem:UCount} and Theorem~\ref{thm:TCount}, we have
\begin{align*}
\lambda\left(\frac{k[X]}{\mathfrak{m}^{sq}+\mathfrak{m}^{[q]}+I_2(X)}\right) =&\; T(m,n,sq,q)+T(n,m,sq,q)-U(m,n,sq,q)\\
=&\sum_{i}\sum_{a}\sum_{b}\sum_{\ell}(-1)^i \binom{m}{i}\binom{m-1}{m-1-a}\binom{iq+n-1}{n-1-b}\binom{(s-i)q+\ell}{a+b+1}\binom{a}{\ell}\binom{b}{\ell}
\\
&+\sum_{j}\sum_{a}\sum_{b}\sum_{\ell}(-1)^j \binom{n}{j}\binom{jq+m-1}{m-1-a}\binom{n-1}{n-1-b}\binom{(s-j)q+\ell}{a+b+1}\binom{a}{\ell}\binom{b}{\ell}\\
&-\sum_i\sum_{j}\sum_{a}\sum_{b}\sum_{\ell}(-1)^{i+j}\binom{m}{i}\binom{n}{j}\binom{(j\dotdiv i)q+m-1}{m-1-a}\binom{(i \dotdiv j)q+n-1}{n-1-b}\binom{(s-\max\{i,j\})q+\ell}{a+b+1}\binom{a}{\ell}\binom{b}{\ell}\\
=&\sum_{a}\sum_{b}\sum_{\ell}\binom{m-1}{a}\binom{n-1}{b}\binom{sq+\ell}{a+b+1}\binom{a}{\ell}\binom{b}{\ell}\\
&-\sum_{i>0}\sum_{j>0}\sum_{a}\sum_{b}\sum_{\ell} (-1)^{i+j}\binom{m}{i}\binom{n}{j}\binom{(j\dotdiv i)q+m-1}{m-1-a}\binom{(i \dotdiv j)q+n-1}{n-1-b}\binom{(s-\max\{i,j\})q+\ell}{a+b+1}\binom{a}{\ell}\binom{b}{\ell}
\end{align*} where the last equality follows as the $i = 0$ summand of $U(m,n,sq,q)$ is precisely $T(n,m,sq,q)$ and the $j = 0$ summand of $U(m,n,sq,q)$ is precisely $T(m,n,sq,q)$ and so they cancel in the sum, except for the summand $\sum_{a}\sum_{b}\sum_{\ell}\binom{m-1}{a}\binom{n-1}{b}\binom{sq+\ell}{a+b+1}\binom{a}{\ell}\binom{b}{\ell}$ which only appears once in $U(m,n,sq,q)$ when $i = 0$ and $j = 0$, but appears twice in $T(m,n,sq,q) + T(n,m,sq,q)$. 

\end{proof}

\subsection{Examples}

Fix $R = k[X]/I_2(X)$ where $X$ is an $m \times n$-matrix. We conclude with a few examples using Theorem~\ref{thm:TCount} and use this to calculate $e_s(R)$ for small values of $m$ and $n$.

\begin{xmp}

Suppose $m = n = 2 $ and assume throughout assume $s \in \ZZ[p^{-1}]$ and so $sq$ is always an integer for $q \gg 0$. We calculate $ \lambda\left(\frac{k[X]}{\mathfrak{m}^{sq}+\mathfrak{m}^{[q]}+I_2(X)}\right)$ by calculating $R(2,2,s,q)$ and $S(2,2,s,q)$. The latter depends on the integer part of $s$. 

We always have $R(2,2,s,q) = \frac{s^3 q^3}{3} + \frac{s^2q^2}{2} + \frac{sq}{6}$. For $s < 1$, $S(2,2,s,q) = 0$ and so $ \lambda\left(\frac{k[X]}{\mathfrak{m}^{sq}+\mathfrak{m}^{[q]}+I_2(X)}\right) = \frac{s^3 q^3}{3} + \frac{s^2q^2}{2} + \frac{sq}{6}$. For $1 \leq s < 2$ we have $$S(2,2,s,q) =  \frac{4}{3} (s-1)^3 q^3 + 2 (s-1)^2 q^2 + \frac{2}{3} (s-1) q.$$ Likewise, for $s \geq 2$ we have $$S(2,2,s,q) =  \left( \frac{s^3}{3} - \frac{4}{3} \right) q^3 + \frac{s^2}{2} q^2 +  \left( \frac{s}{6} + \frac{1}{3} \right) q .$$ Putting this together yields a closed form for the length in question

$$ \lambda\left(\frac{k[X]}{\mathfrak{m}^{sq}+\mathfrak{m}^{[q]}+I_2(X)}\right)  = 
\begin{cases} \frac{s^3 }{3} q^3 + \frac{s^2}{2} q^2 + \frac{s}{6} q & \text{if } 0 < s \leq 1 \\ 
\left(\frac{s^3 }{3} - \frac{4}{3} (s-1)^3 \right) q^3 +  \left( \frac{s^2}{2} - 2 (s-1)^2 \right) q^2 +  \left( \frac{s}{6} - \frac{2}{3} (s-1) \right) q & \text{if } 1 < s \leq 2 \\   
\frac{4}{3}  q^3  - \frac{1}{3}  q & \text{if } s > 2 \end{cases}.$$ 
Recalling that $$2 \cH_s(3) - 2\cH_{s-1}(3) = \begin{cases}  \frac{1}{3} s^3 & \text{if } s < 1 \\ \frac{1}{3} s^3 - \frac{4}{3} (s-1)^3 & \text{if } 1 \leq s < 2 \end{cases}$$ we have shown that $$h_s(\fm) = \lim_{q \to \infty}  \frac{1}{q^3} \lambda\left(\frac{k[X]}{\mathfrak{m}^{ sq }+\mathfrak{m}^{[q]}+I_2(X)}\right)  = \begin{cases} 2 \cH_s(3) - 2\cH_{s-1}(3) & \text{if } s\leq 2 \\ \frac{4}{3} & \text{if } s\geq 2.\end{cases}$$ and thus $$e_s(R)= \begin{cases} 2-\frac{2\cH_{s-1}(3)}{\cH_s(3)} & \text{if } s\leq 2 \\ \frac{4}{3\cH_s(3)} & \text{if } s\geq 2.\end{cases}.$$

\end{xmp}

\begin{xmp}


Now fix $m = 2$ and $n = 3$. We have $R(2,3,s,q) = \frac{s^4}{8} q^4 + \frac{5 s^3}{12} q^3 + \frac{3 s^2}{8} q^2 + \frac{s}{12} q$ and $$S(2,3,s,q)  =  \begin{cases} 
0 & \text{if } 0 < s < 1 \\ 
\frac{3}{4} (s-1)^4 q^4 + \frac{5}{2} (s-1)^3 q^3 + \frac{9}{4} (s-1)^2 q^2 + \frac{1}{2} (s-1) q & \text{if } 1 \leq s < 2 \\   
\frac{1}{4} ( 4s^3 - 9s^2 + 7) q^4 + \frac{1}{4} ( 9s^2 - 9s - 8) q^3 + \frac{1}{4} (5s-1) q^2 + \frac{1}{2} q  & \text{if } 2 \leq s < 3 \\
(\frac{-13}{8} +\frac{s^4}{8}) q^4  + (  \frac{1}{4} + \frac{5}{12} s^3 ) q^3 + (\frac{1}{8} + \frac{3}{8} s^2) q^2 + ( \frac{1}{4} + \frac{s}{12} ) q   & \text{if } 3 \leq s \end{cases}$$ and we have 
the length $ \lambda\left(\frac{k[X]}{\mathfrak{m}^{sq}+\mathfrak{m}^{[q]}+I_2(X)}\right)$ is given by $$\begin{cases} 
\frac{s^4}{8} q^4 + \frac{5 s^3}{12} q^3 + \frac{3 s^2}{8} q^2 + \frac{s}{12} q & \text{if } 0 < s < 1 \\ 
( \frac{s^4}{8} - \frac{3}{4} (s-1)^4 )q^4 + ( \frac{5 s^3}{12} - \frac{5}{2} (s-1)^3 ) q^3 + ( \frac{3 s^2}{8} - \frac{9}{4} (s-1)^2 )q^2 +  (\frac{s}{12}  - \frac{1}{2} (s-1) ) q & \text{if } 1 \leq s < 2 \\   
( \frac{s^4}{8} - \frac{1}{4} ( 4s^3 - 9s^2 + 7) ) q^4 + ( \frac{5 s^3}{12} - \frac{1}{4} ( 9s^2 - 9s - 8) ) q^3 + (\frac{3 s^2}{8} -  \frac{1}{4} (5s-1) ) q^2 + ( \frac{s}{12}  - \frac{1}{2} )q  & \text{if } 2 \leq s < 3 \\
\frac{13}{8}  q^4  -  \frac{1}{4}  q^3 - \frac{1}{8} q^2 - \frac{1}{4} q   & \text{if } 3 \leq s \end{cases}.$$ 

We have that $$h_s(\fm) = \lim_{q \to \infty} \frac{1}{q^4}  \lambda\left(\frac{k[X]}{\mathfrak{m}^{sq}+\mathfrak{m}^{[q]}+I_2(X)}\right) = \begin{cases} 
\frac{s^4}{8}  & \text{if } 0 < s < 1 \\ 
( \frac{s^4}{8} - \frac{3}{4} (s-1)^4 )  & \text{if } 1 \leq s < 2 \\   
( \frac{s^4}{8} - \frac{1}{4} ( 4s^3 - 9s^2 + 7) )   & \text{if } 2 \leq s < 3 \\
\frac{13}{8}  & \text{if } 3 \leq s \end{cases}.$$ In terms of the normalizing factors $\cH_s$ one may write this as 

\[e_s(R)=\begin{cases} 3 & \text{if } 0 < s < 1 \\ 3-\frac{6\cH_{s-1}(4)}{\cH_s(4)} & \text{if } 1\leq s < 2\\ \frac{13}{8\cH_s(4)}-\frac{\cH_{3-s}(4)+s\cH_{3-s}(3)}{\cH_s(4)} & \text{if } 2\leq s < 3 \\ \frac{13}{8\cH_s(4)} &\text{if } 3 \leq s \end{cases}.\]

\end{xmp}

\end{document}